\documentclass[11pt,a4paper,english,reqno]{amsart}
\usepackage{amsmath,amsthm,amssymb}
\usepackage{calrsfs}
\usepackage[T1]{fontenc}
\usepackage{lmodern}
\usepackage{microtype}
\usepackage[hidelinks]{hyperref}
\hypersetup{pdftitle={A counterexample to the Polishchuk--Van den Bergh conjecture},pdfauthor={Shengyong Pan}}

\theoremstyle{plain}
\newtheorem{thm}{Theorem}[section]
\newtheorem{prp}[thm]{Proposition}
\newtheorem{lem}[thm]{Lemma}
\newtheorem{cor}[thm]{Corollary}
\newtheorem{cnj}{Conjecture}

\newtheorem{thmmain}{Theorem}

\theoremstyle{definition}
\newtheorem*{dfn-nn}{Definition}
\newtheorem{rmk}[thm]{Remark}
\newcommand{\thmref}[1]{Theorem~\ref{#1}}
\newcommand{\prpref}[1]{Proposition~\ref{#1}}
\newcommand{\lemref}[1]{Lemma~\ref{#1}}
\newcommand{\cnjref}[1]{Conjecture~\ref{#1}}

\def\al{\alpha}
\def\de{\delta}
\def\ep{\varepsilon}
\def\ze{\zeta}
\def\et{\eta}
\def\si{\sigma}
\def\om{\omega}
\def\De{\Delta}

\def\Hom{\operatorname{Hom}}
\def\Ext{\operatorname{Ext}}
\def\RHom{\operatorname{RHom}}
\def\Coh{\operatorname{Coh}}
\def\Perf{\operatorname{Perf}}
\def\oc{\operatorname{oc}}
\def\Spec{\operatorname{Spec}}
\def\genus{\operatorname{genus}}
\def\rep{\operatorname{rep}}
\def\mod{\operatorname{mod}}
\def\add{\operatorname{add}}
\def\rad{\operatorname{rad}}
\def\rank{\operatorname{rank}}
\def\curExt{{\operatorname{\mathcal{E}\mathnormal{xt}}}}
\def\curEnd{{\operatorname{\mathcal{E}\mathnormal{nd}}}}

\def\Db{{\mathcal D}^{\text{\rm b}}}
\def\calA{{\mathcal A}}
\def\calB{{\mathcal B}}
\def\calI{{\mathcal I}}
\def\calO{{\mathcal O}}
\def\calT{{\mathcal T}}
\def\calU{{\mathcal U}}

\def\bbZ{{\mathbb Z}}
\def\bbP{{\mathbb P}}
\def\bbQ{{\mathbb Q}}
\def\bbC{{\mathbb C}}

\def\inv{^{-1}}
\def\incl{\hookrightarrow}
\def\iso{\cong}
\def\ox{\otimes}
\def\Lox{\overset{\mathbf{L}}{\otimes}}
\def\Ds{\bigoplus}
\def\Sm{\sum}
\def\bs{\backslash}
\def\ya#1{\xrightarrow{#1}}
\def\pt{\Spec\bbC}

\numberwithin{equation}{section}

\title[A counterexample on curves]{A counterexample to the Polishchuk--Van den Bergh conjecture}
\author{Shengyong Pan}
\address{School of Mathematics and Statistics; Beijing Key Laboratory of
Biological Big Data and Topological Statistics, Beijing Jiaotong University,
Beijing 100044, China.}
\email{shypan@bjtu.edu.cn}
\date{}

\begin{document}
\begin{abstract}
We construct an effective $S_3$ action on the smooth projective model
of $y^2=t^{12h-5}+t$ for every $h\geq1$. Its quotient has genus $h$,
and the fixed quotient for the three-cycle class consists of two points.
We prove that the derived category of these two points admits no
$\bbC$-linear exact fully faithful embedding into the equivariant derived
category. This gives a counterexample to Conjecture A of Polishchuk
and Van den Bergh. More generally, for a positive-genus quotient we
compute the maximum number of completely orthogonal exceptional
objects as $\Sm_i\lfloor e_i/2\rfloor$, where $e_i$ are the inertia
orders. We also give a necessary and sufficient condition for
conjugacy-class decompositions with embeddings linear over the
quotient, and construct kernels on the natural fibre products.
For genus-zero quotients we obtain Euler-form bounds and exact
capacities for all-two ramification with at least four branch points
and for types $(2,3,3)$, $(2,3,4)$ and $(2,3,5)$.
\end{abstract}
\maketitle
\tableofcontents

\section*{Introduction}

Throughout this paper all schemes, representations and functors are over
$\bbC$. For a variety $X$ with an action of a finite group $G$, we write
$\Db(X)=\Db(\Coh(X))$ and
$\Db_G(X)=\Db(\Coh_G(X))=\Db([X/G])$.
For $g\in G$, let $X^g$ be its fixed locus, $C_G(g)$ its centralizer,
and $[g]$ its conjugacy class. The derived category of the empty scheme
is understood to be zero.

The fixed loci of a finite group action give decompositions of several
additive invariants of the equivariant derived category. When the
geometric quotients $X^g/C_G(g)$ are smooth, their derived categories
are natural candidates for the components of a semiorthogonal
decomposition. Polishchuk and Van den Bergh formulated this question
as follows; see \cite[Conjecture~A, p.~2654]{PV}.

\begin{cnj}[Polishchuk--Van den Bergh, Conjecture A]
\label{conj:original}
Let a finite group $G$ act effectively on a smooth
quasiprojective variety $X$. Suppose that every geometric
quotient $X^g/C_G(g)$ is smooth. Then $\Db_G(X)$ admits a
semiorthogonal decomposition with one subcategory
$\calA_{[g]}$ for each conjugacy class, such that
\[
 \calA_{[g]}\simeq\Db(X^g/C_G(g)).
\]
\end{cnj}

The effectivity assumption is essential: Bergh, Gorchinskiy, Larsen
and Lunts constructed a non-effective action on an elliptic curve
for which the conclusion fails \cite[Theorem~6.20]{BGLL}.
The conjecture has also been studied for linear reflection-group
actions; see \cite{PV,Krug}. In this paper we construct counterexamples
with effective actions on smooth projective curves.

Let us first explain the question for curves. For an effective action
on a smooth connected curve $C$, the nonidentity fixed quotients are
finite sets of points. The known curve decomposition consists of
$\Db(C/G)$ and exceptional objects supported on the branch orbits
\cite[Theorem~1.2]{Polishchuk}\cite[Theorem~4.3.1, p.~2705]{PV}.
To obtain a block equivalent to the category of $n$ points, however,
one needs $n$ exceptional objects whose Hom spaces vanish in every
degree in both directions. An exceptional sequence imposes vanishing
in only one direction. Thus the relevant question is how many
completely orthogonal exceptional objects can occur in the whole
equivariant derived category.

Our first main result gives an explicit family for which the required
two-point category does not embed at all.

\begin{thmmain}\label{thm:main}
For every integer $h\geq1$, let $C_h$ be the smooth
projective model of
\begin{equation}\label{eq:family}
 y^2=t^{12h-5}+t.
\end{equation}
There is an effective $S_3$ action on $C_h$ for which:
\begin{enumerate}
\item $\genus(C_h)=6h-3$ and $\genus(C_h/S_3)=h$;
\item there is exactly one nonfree orbit, with stabilizer
      of order three;
\item the fixed quotient for the three-cycle class is
      $\pt\amalg\pt$, and that for the transposition class
      is empty;
\item there is no $\bbC$-linear exact fully faithful functor
      \[
       \Db(\pt\amalg\pt)\longrightarrow\Db_{S_3}(C_h).
      \]
\end{enumerate}
Consequently, all the smoothness and effectivity hypotheses
of \cnjref{conj:original} hold, but its conclusion
fails.
\end{thmmain}

For $h=1$ this gives the genus-three curve $y^2=t^7+t$, whose quotient
is the elliptic curve $b^2=u^3-4u$. The conclusion does not assume
that a putative embedding commutes with tensoring from the quotient.
It therefore excludes the required component independently of any
choice of kernels or semiorthogonal decomposition.

To prove this result, we compute the maximum number of completely
orthogonal exceptional objects. We use the following notation.

\begin{dfn-nn}\label{def:capacity}
An object $E$ of a $\bbC$-linear triangulated category is
\emph{exceptional} if $\RHom(E,E)\simeq\bbC$.
Exceptional objects $E_1,\ldots,E_m$ are
\emph{completely orthogonal} if
\[
 \Hom(E_i,E_j[q])=0
 \qquad(q\in\bbZ,\ i\ne j).
\]
Write $\oc(\calT)$ for the supremum of the lengths
of such collections in $\calT$, allowing the empty
collection.
\end{dfn-nn}

Our second main result determines this number whenever the quotient
curve has positive genus.

\begin{thmmain}\label{thm:capacity}
Let $G$ act effectively on a smooth connected projective
complex curve $C$, and put $Y=C/G$. Let $e_1,\ldots,e_b$
be the inertia orders at the branch points.
If $\genus(Y)\geq1$, then
\begin{equation}\label{eq:capacity}
 \oc\bigl(\Db_G(C)\bigr)
 =
 \Sm_{i=1}^b\left\lfloor\frac{e_i}{2}\right\rfloor.
\end{equation}
In particular, a $\bbC$-linear exact fully faithful functor
$\Db(\bbC^n)\longrightarrow\Db_G(C)$ exists if and only if
$n$ does not exceed the right-hand side of
\eqref{eq:capacity}.
\end{thmmain}

The proof has two parts. First, positive genus of $C/G$ excludes
exceptional vector bundles by the trace map, and heredity reduces
exceptional complexes to torsion sheaves supported on single branch
orbits. Second, for cyclic inertia of order $e$, the symmetric Euler
pairing identifies the successive differences of an exceptional
class with a vector $\ep_a-\ep_b$. Complete
orthogonality forces the endpoint sets $\{a,b\}$ to be disjoint.
This gives the bound $\lfloor e/2\rfloor$, attained by alternating
character simples. In the family above there is one branch orbit
of inertia order three, so the capacity is one although the
three-cycle fixed quotient has two points.

We next consider embeddings compatible with the quotient map
$\pi:C\to Y=C/G$. Such an embedding is called $\Perf(Y)$-linear
if it commutes with the actions of perfect complexes on $Y$, with
the usual unit and associativity compatibilities. The role of
kernels on fibre products over $Y$ also appears in the formulation
of motivic semiorthogonal decompositions
\cite[Definition~1.1.1]{LP}. For linear embeddings the obstruction
can be tested separately at each branch orbit. Let $H_i$ be an
inertia group at the $i$th branch orbit and put $e_i=|H_i|$.
Our third main result is the following criterion.

\begin{thmmain}
\label{thm:fusioncriterion}
Let a finite group $G$ act effectively on a smooth connected
quasiprojective complex curve $C$. The following conditions
are equivalent:
\begin{enumerate}
\item \cnjref{conj:original} admits a realization
whose embeddings are $\Perf(Y)$-linear for the natural
maps of the fixed quotients to $Y=C/G$.
\item For every branch inertia group $H_i$ and every
nonidentity conjugacy class,
\begin{equation}\label{eq:fusioncriterion}
 |[g]\cap H_i|
 \leq\left\lfloor\frac{e_i}{2}\right\rfloor.
\end{equation}
\end{enumerate}
When these conditions hold, the identity block may be chosen
as $\pi^*\Db(Y)$, and the nonidentity blocks may occur in any
prescribed total order. All embeddings can be given by
$G$-equivariant kernels on the natural fibre products
\[
 (C^g/C_G(g))\times_Y C.
\]
\end{thmmain}

For necessity, linearity keeps the image of each source point over
the same point of $Y$, so the local orthogonality bound applies.
For sufficiency, the branch subcategory is of type $A_{e_i-1}$.
We colour a path with the prescribed multiplicities and orient its
edges according to a prescribed order of the conjugacy classes.
The corresponding simple modules give the required semisimple
blocks. Derived reflection equivalences transport these blocks to
the branch subcategory. A bounded representative in the standard
branch generators then gives a kernel on the required fibre
product. This criterion does not require $C$ to be projective or
$Y$ to have positive genus.

Two further consequences are discussed in the paper. First, the
known ordered curve decomposition can always be indexed by the
connected components of the fixed quotients. This statement allows
components of a single conjugacy class to occur at different places
in the order; related connected-inertia decompositions occur in
\cite[Theorem~1.1]{Lim}. For a linear action on a vector space each
fixed quotient is irreducible, so this distinction disappears.
Second, when $Y\simeq\bbP^1$, exceptional vector bundles must
be included in the capacity calculation. We obtain bounds from
both Euler forms, prove that the capacity is $b$ for $b\geq4$
branch orbits of inertia order two, and compute the capacities
$4,5,5$ for types $(2,3,3)$, $(2,3,4)$ and $(2,3,5)$.
The latter calculations follow from a formula for hereditary
categories whose underlying quiver is a tree: the capacity equals
the maximum independent set size.

The paper is organized as follows. In Section~\ref{sec:capacity}
we compute the local Euler pairing and prove
\thmref{thm:capacity}. In Section~\ref{sec:family} we construct
the $S_3$ actions and prove \thmref{thm:main}.
In Section~\ref{sec:refinement} we recall the ordered curve
decomposition, identify the branch categories of type $A$, and
discuss connected fixed quotients. We
proves \thmref{thm:fusioncriterion} in Section~\ref{sec:fusion}.
In Section~\ref{sec:literature} we treat genus-zero quotients and
tree models. The remaining capacity and realization questions
are collected in Section~\ref{sec:questions}.

\subsection*{Conventions}
The notation $\langle\calA_1,\ldots,\calA_s\rangle$
means that the indicated full triangulated subcategories generate
the ambient category and
$\Hom(\calA_j,\calA_i[q])=0$ for $j>i$ and all $q\in\bbZ$.
For objects with finite-dimensional total Ext, we write
$\chi(E,F)=\Sm_q(-1)^q\dim_\bbC\Hom (E,F[q])$;
the same notation is used for its pairing on the Grothendieck group.
Branch orbits carry their reduced scheme structures. We write
$\pt=\Spec \bbC$ and use $\Db(\bbC^n)$ for the derived
category of finite-dimensional modules over $\bbC^n$.

\section*{Acknowledgments}
The mathematical ideas in this paper are the author's own.
The author used ChatGPT, 
to help draft and revise the manuscript. The author is responsible
for the mathematical content of the paper. This work was supported
by the Beijing Natural Science Foundation (Grant Nos. 1262017, 1252011).

\section{Exceptional objects and the local Euler pairing}
\label{sec:capacity}

We first locate exceptional objects when the quotient has positive
genus. We then compute the Euler pairing at a branch orbit and
deduce \thmref{thm:capacity}. The local computation will also
be used without a genus assumption in Sections~\ref{sec:fusion}
and \ref{sec:literature}.

\subsection{Support of exceptional objects}
For a finite effective action on a smooth connected curve, every
stabilizer is cyclic and acts faithfully on the tangent line.
Indeed a finite-order automorphism fixing a point and having
derivative one is the identity in characteristic zero. Thus the
stabilizer embeds into $\bbC^\times$. There are only finitely many
nonfree orbits.

\begin{lem}\label{lem:torsion}
Under the hypotheses of \thmref{thm:capacity}, every exceptional
object of $\Db_G(C)$ is a shift of an indecomposable finite-length
equivariant sheaf supported on one nonfree orbit.
\end{lem}
\begin{proof}
The category $\Coh_G(C)$ is hereditary. For coherent sheaves $F,H$,
one has $\Ext^q_G(F,H)=\Ext^q_C(F,H)^G$; invariants are exact over
$\bbC$. On a smooth curve the latter Ext groups vanish for $q>1$.
For example, this follows from the local-to-global spectral sequence:
the sheaf $\curExt^1(F,H)$ has finite support, and higher local
Ext and second sheaf cohomology vanish. In a hereditary category
every bounded complex splits as the direct sum of its shifted
cohomology sheaves, by the truncation triangles. An exceptional
object has no nontrivial direct-sum decomposition, so it is a shift
of one indecomposable sheaf.

A coherent equivariant sheaf splits as a torsion sheaf and a vector
bundle. In fact its torsion-free quotient $V$ is locally free, and
for a finite-length sheaf $T$,
$\Ext^1_G(V,T)=H^1(C,V^\vee\ox T)^G=0$.
If $V$ has positive rank, the identity and normalized trace split
$\calO_C$ off $\curEnd(V)$ equivariantly. Therefore
\[
 H^1(Y,\calO_Y)
 \iso H^1(C,\calO_C)^G
 \incl \Ext^1_G(V,V).
\]
The first isomorphism follows from the finite quotient map and exact
invariants. Since $\genus(Y)\geq1$, a vector bundle cannot be
exceptional.

A torsion sheaf splits equivariantly over the $G$-orbits in its
support. Indecomposability leaves just one orbit. At a free orbit,
equivariant finite-length sheaves are equivalent to finite-length
modules over a discrete valuation ring. Its unique simple module
has one-dimensional Hom and Ext$^1$ with itself, so the Euler
pairing is zero on the finite-length Grothendieck group. No object
supported on a free orbit can consequently be exceptional.
\end{proof}

\subsection{The local Euler pairing}
Fix a nonfree orbit $D$ with inertia group $H$ of order $e\geq2$.
Equivariant completion along $D$ identifies its finite-length category
with finite-length modules over
$A=R\rtimes H$, where $R=\bbC[[t]]$. Choose a generator $\si$ of
$H$ and a primitive $e$th root $\ze$ so that $\si(t)=\ze t$.
Such a parameter exists by averaging a local parameter with the
appropriate character. Induction from a chosen point of $D$ gives
the equivalence with the orbit category.

These identifications preserve Ext. For finite-support sheaves the
sheaf Ext groups have finite support and no higher cohomology;
completion is flat and preserves Ext between the corresponding finite
modules. Hence computations over $A$ give the global equivariant
Ext groups for objects supported on $D$.

Let $\al(\si)=\ze$, and let $S_i$ be the simple module
with character $\al^i$, for $i\in\bbZ/e\bbZ$. An equivariant
resolution is
\[
 0\longrightarrow R\ox\al^{i+1}
 \ya{\ t\ }R\ox\al^i
 \longrightarrow S_i\longrightarrow0.
\]
It gives
\begin{equation}\label{eq:simples}
 \Hom_A(S_i,S_j)=\bbC^{\de_{ij}},\quad
 \Ext^1_A(S_i,S_j)=\bbC^{\de_{j,i+1}},\quad
 \Ext^q_A(S_i,S_j)=0\quad(q>1).
\end{equation}
All indices in \eqref{eq:simples} are taken modulo $e$; the notation
$\bbC^0$ means the zero vector space.

The classes of the $S_i$ form a basis of the finite-length
Grothendieck group. For $v,w\in\bbZ^e$, equation~\eqref{eq:simples}
implies
\[
 \chi(v,w)=\Sm_i v_iw_i-\Sm_i v_iw_{i+1}.
\]
Define $\De v=(v_i-v_{i+1})_i$. Then
\begin{equation}\label{eq:euler}
 \chi(v,v)=\tfrac12\|\De v\|^2,
 \qquad
 \chi(v,w)+\chi(w,v)=\De v\cdot\De w.
\end{equation}

\begin{lem}\label{lem:localbound}
Among the exceptional objects supported on $D$, the maximum number
that can be pairwise completely orthogonal is $\lfloor e/2\rfloor$.
\end{lem}
\begin{proof}
If $v$ is the class of an exceptional object, then $\chi(v,v)=1$.
By \eqref{eq:euler}, the integral vector $\De v$ has squared
length two and coordinate sum zero. Thus
$\De v=\ep_a-\ep_b$ for two distinct indices
$a,b$, where $\ep_i$ is a standard basis vector of $\bbZ^e$.

Two vectors of this form have zero dot product exactly when their
two-element endpoint sets are disjoint: identical endpoint sets give
dot product $2$ or $-2$, and exactly one common endpoint gives $1$
or $-1$. Complete orthogonality implies zero symmetric Euler pairing,
so \eqref{eq:euler} forces disjoint endpoint sets. A collection of
length $m$ therefore uses $2m$ distinct indices, proving $2m\leq e$.

The bound is attained. The simples with indices
$0,2,\ldots,2\lfloor e/2\rfloor-2$ are exceptional, and no two
chosen indices are adjacent in the cyclic order. Equation~\eqref{eq:simples}
shows complete orthogonality. Inducing these modules to $D$ gives
the desired equivariant sheaves.
\end{proof}

\begin{proof}[Proof of \thmref{thm:capacity}]
By \lemref{lem:torsion}, each exceptional object lies on one
nonfree orbit. \lemref{lem:localbound} gives the upper bound at
each orbit. Objects with disjoint supports are completely orthogonal,
so choosing the local collections just constructed attains the sum
in \eqref{eq:capacity}.

The images of the $n$ simple objects of $\Db(\bbC^n)$ under any exact
fully faithful functor are completely orthogonal exceptional objects.
This gives necessity. Conversely, any collection of $n$ such objects
defines a fully faithful functor: on the $i$th summand of
$\Db(\bbC^n)$ send a bounded complex of vector spaces $V$ to
$V\otimes_\bbC E_i$, and take the direct sum. This proves sufficiency.
\end{proof}

\begin{cor}\label{cor:necessary}
For a curve as in \thmref{thm:capacity}, let
$n_g=\#(C^g/C_G(g))$ for $g\ne1$. If \cnjref{conj:original}
holds, then $n_g\leq\Sm_i\lfloor e_i/2\rfloor$ for every such $g$.
\end{cor}
\begin{proof}
Each required subcategory supplies a fully faithful embedding of
$\Db(\bbC^{n_g})$. Apply \thmref{thm:capacity}.
\end{proof}

The genus assumption is needed only to exclude exceptional vector
bundles. \lemref{lem:localbound} is a local statement with no
assumption on the genus of the quotient.

\section{Proof of the counterexample}
\label{sec:family}

We now construct the actions in \thmref{thm:main}.
Fix $h\geq1$ and a primitive cube root $\om$. On the function
field of \eqref{eq:family}, consider
\begin{equation}\label{eq:action}
 r(t,y)=(\om t,\om^2y),\qquad
 s(t,y)=(t\inv ,-y/t^{6h-2}).
\end{equation}

\begin{prp}\label{prop:action}
The maps in \eqref{eq:action} extend to automorphisms of $C_h$ and
generate an effective $S_3$ action. The only nonfree orbit is
$D=\{P_0,P_\infty\}$, where $P_0$ and $P_\infty$ lie above
$t=0$ and $t=\infty$. Its stabilizers are $\langle r\rangle$.
\end{prp}
\begin{proof}
Put $N=12h-5$ and $\ell=6h-2$. The polynomial $t^N+t$ is
squarefree: its derivative is $1$ at zero, and at a nonzero root
$t^{N-1}=-1$ the derivative is $1-N\ne0$. Since $N$ is odd,
the smooth projective hyperelliptic curve has genus
$(N-1)/2=6h-3$ and a unique point at infinity.

The congruence $N\equiv1\pmod3$ shows that $r$ preserves the
equation. The identity
$(t^N+t)/t^{N+1}=t^{-N}+t\inv $ and the equality $2\ell=N+1$
show the same for $s$. These function-field automorphisms extend
uniquely to the smooth projective model. They satisfy $r^3=s^2=1$.
For the remaining relation, substitution gives
\[
 srs(t,y)=(\om\inv t,\om^{2-\ell}y)
          =(\om\inv t,\om y)=r\inv (t,y),
\]
because $\ell\equiv1\pmod3$. The transformations $r^i$ and $r^is$
for $0\leq i<3$ are distinct already on $t$. Thus the action is
effective and the generated group is $S_3$.

The maps $r$ and $r^2$ fix exactly $P_0$ and $P_\infty$.
The involution $s$ exchanges these points. A fixed point of $s$
elsewhere would have $t=1$ or $t=-1$. Since $\ell$ is even, $s$
negates $y$ at these two values, but $y^2=2$ and $y^2=-2$,
respectively. It has no fixed point. All transpositions are
conjugate to $s$, so they have no fixed points either. This proves
the asserted stabilizers and exhausts the nonfree locus.
\end{proof}

\begin{prp}\label{prop:quotient}
The quotient $Y_h=C_h/S_3$ has genus $h$. More explicitly, it is
the smooth projective model of
\begin{equation}\label{eq:quotient}
 b^2=(u^2-4)P_{2h-1}(u),
\end{equation}
where $P_0(u)=2$, $P_1(u)=u$, and
$P_{n+1}(u)=uP_n(u)-P_{n-1}(u)$.
\end{prp}
\begin{proof}
Set $x=t^3$ and $v=ty$. They are $r$-invariant and satisfy
$v^2=x^{4h-1}+x$. They generate the fixed field of
$\langle r\rangle$: adjoining $t$ recovers $y=v/t$ and has degree
three. On this intermediate field, $s$ sends
$(x,v)$ to $(x\inv ,-v/x^{2h})$.

Now put $u=x+x\inv $ and $b=v x^{-h}(x-x\inv )$. Both are
$s$-invariant. The recurrence gives
$P_n(x+x\inv )=x^n+x^{-n}$, so their relation is
\eqref{eq:quotient}. Conversely,
$x^2-ux+1=0$ and $v=b x^h/(x-x\inv )$. The extension of
$\bbC(u,b)$ obtained by adjoining $x,v$ therefore has degree at most
two, and exactly two because $s$ is nontrivial. Consequently
$\bbC(u,b)=\bbC(C_h)^{S_3}$.

For $n=2h-1$, the identity
$P_n(2\cos\theta)=2\cos(n\theta)$ gives $n$ distinct roots
strictly between $-2$ and $2$. Also $P_n(2)=2$ and $P_n(-2)=-2$.
Thus the right-hand side of \eqref{eq:quotient} is squarefree of
degree $2h+1$, and its smooth projective model has genus $h$.
The same genus follows from Riemann--Hurwitz using
\prpref{prop:action}:
\[
 2(6h-3)-2=6\bigl(2\genus(Y_h)-2\bigr)+2(3-1).
\]
\end{proof}

\begin{proof}[Proof of \thmref{thm:main}]
Propositions~\ref{prop:action} and \ref{prop:quotient} establish
the action, genera, and stabilizers. The three conjugacy classes
are represented by $1,r,s$. Their fixed quotients are
\[
\begin{array}{c|c|c}
 g & C_h^g & C_h^g/C_{S_3}(g)\\ \hline
 1 & C_h & Y_h\\
 r & \{P_0,P_\infty\} & \pt\amalg\pt\\
 s & \varnothing & \varnothing.
\end{array}
\]
Indeed $C_{S_3}(r)=\langle r\rangle$ fixes the two points
individually. These are smooth geometric quotients. There is one
branch point of $Y_h$, with inertia order three, and
$\genus(Y_h)=h\geq1$. \thmref{thm:capacity} gives
$\oc(\Db_{S_3}(C_h))=1$. The derived category of two disjoint
points would require two completely orthogonal exceptional objects,
which is impossible.
\end{proof}

\begin{rmk}\label{rem:abstract}
The conclusion rules out every exact fully faithful embedding of
the specified point category. In particular, changing its support,
choosing another coarse-curve subcategory, or applying mutations
cannot yield the decomposition in \cnjref{conj:original}.
No compatibility with $Y_h$ was assumed in this obstruction.
\end{rmk}

\section{Decompositions for curve actions}
\label{sec:refinement}

\subsection{The ordered curve decomposition}
Let $G$ act effectively on a smooth connected quasiprojective curve
$C$, let $\pi:C\to Y=C/G$, and let $D_1,\ldots,D_b$ be the
nonfree orbits. Denote their images by $y_i$ and their inertia orders
by $e_i$. For $1\leq l\leq e_i$, define
\[
 M_{i,l}=\calO_{lD_i}
        =\calO_C/\calO_C(-lD_i),
\]
with the natural equivariant structure.

We will use the ordered curve decomposition of
\cite[Theorem~1.2]{Polishchuk}\cite[Theorem~4.3.1, p.~2705]{PV}.
The following form also records the endomorphism algebra of each
branch collection, which is needed in Section~\ref{sec:fusion}.

\begin{prp}\label{prop:ordered}
There is a semiorthogonal decomposition
\begin{equation}\label{eq:ordered}
 \Db_G(C)=
 \left\langle\pi^*\Db(Y),\calB_1,\ldots,\calB_b\right\rangle,
 \qquad
 \calB_i=\left\langle M_{i,e_i-1},\ldots,M_{i,1}\right\rangle.
\end{equation}
The $\calB_i$ are pairwise completely orthogonal. Each indicated
collection in $\calB_i$ is strong exceptional, and
$\calB_i\simeq\Db(\mod\,\bbC A_{e_i-1})$, where
$A_{e_i-1}$ is a linearly oriented quiver of type $A$.
\end{prp}
\begin{proof}
The cited curve decomposition gives \eqref{eq:ordered} and the
orthogonality of the branch categories. We verify the strong
exceptional property and their type-$A$ description.
Fix one orbit, put $e=e_i$, and use the local notation of
Section~\ref{sec:capacity}. The sheaf $M_l$ corresponds to
$R/(t^l)$ with invariant generator and has resolution
\[
0\longrightarrow R\ox\al^l
 \ya{\ t^l\ }R\longrightarrow M_l\longrightarrow0.
\]
For $1\leq l,m<e$, equivariant Hom into $M_m$ gives
\begin{equation}\label{eq:thickenings}
 \Hom_G(M_l,M_m)=
 \begin{cases}\bbC&l\geq m,\\0&l<m,\end{cases}
 \qquad \Ext^q_G(M_l,M_m)=0\quad(q>0).
\end{equation}
Indeed the relevant character spaces are $\bbC\cdot1$ and,
when $l<m$, $\bbC\cdot t^l$; in the latter case the differential
is an isomorphism. When $l\geq m$, the nonzero Hom is generated
by the quotient map $M_l\to M_m$. Its compositions are again
nonzero quotient maps. Thus the descending collection is strong
exceptional and the endomorphism algebra of its direct sum is
the path algebra of a linearly oriented $A_{e-1}$ quiver, up to
opposite. Derived Morita equivalence gives the assertion.
\end{proof}

For the counterexamples, \eqref{eq:ordered} reads
\[
 \Db_{S_3}(C_h)=
 \left\langle\pi^*\Db(Y_h),M_2,M_1\right\rangle,
 \qquad \Hom(M_2,M_1)=\bbC.
\]
The ramification subcategory is $\Db(\mod\,\bbC A_2)$.
It contains an ordered exceptional pair, whereas its two generators
are not completely orthogonal. \thmref{thm:main} shows that
no other pair in the ambient category can supply the missing
$\Db(\bbC^2)$ subcategory either.

\subsection{Counting connected fixed sectors}
Choose a point in $D_i$ and write its stabilizer as $H_i$; thus
$D_i\simeq G/H_i$. For $g\ne1$ there is a bijection, also used in
\cite[Remark~4.3.2, pp.~2705--2706]{PV},
\begin{equation}\label{eq:sectors}
 C_G(g)\bs(G/H_i)^g\ \ya{\ \sim\ }\ [g]\cap H_i,
 \qquad C_G(g)aH_i\longmapsto a\inv ga.
\end{equation}
Indeed $aH_i$ is fixed precisely when $a\inv ga\in H_i$. Since
$H_i$ is abelian, right multiplication of $a$ by an element of
$H_i$ does not change the image. Left multiplication by the
centralizer also does not change it. Surjectivity follows from
conjugacy. If $a\inv ga=b\inv gb$, then $ba\inv \in C_G(g)$,
which proves injectivity.

In particular,
\begin{equation}\label{eq:count}
 n_g=\#(C^g/C_G(g))=\Sm_i|[g]\cap H_i|,
 \qquad
 \Sm_{[g]\ne[1]}n_g=\Sm_i(e_i-1).
\end{equation}
We will use both the number of source points and their images on $Y$
in the proof of \thmref{thm:fusioncriterion}.

\subsection{The connected-component formulation}
For a smooth connected quasiprojective variety $X$ with an effective
finite group action and smooth fixed quotients, set
\[
 \calI(X,G)=\{
 ([g],Z)\mid Z\text{ is a connected component of }X^g/C_G(g)\}.
\]
These are the connected components of the coarse inertia space.

\begin{cnj}[Connected-component formulation]
\label{conj:refined}
There is a total order on $\calI(X,G)$ and a semiorthogonal
decomposition of $\Db_G(X)$ with subcategories
$\calA_{([g],Z)}\simeq\Db(Z)$, one for each pair in
$\calI(X,G)$.
\end{cnj}

This formulation permits components attached to the same conjugacy
class to occur separately in the order. It places no complete
orthogonality requirement between them. It is a weakening of
\cnjref{conj:original}: a block equivalent to the category
of a disconnected smooth scheme can always be refined into the
categories of its connected components. Lim's theorem
\cite[Theorem 1.1]{Lim} gives such a connected-inertia formulation
under the additional assumption of elementary abelian two-group
stabilizers. For curves, its validity follows from the ordered decomposition as follows.

\begin{thm}\label{thm:refinedcurve}
\cnjref{conj:refined} holds for every effective finite group
action on a smooth connected quasiprojective complex curve $C$.
The embedding for the identity component is $\pi^*$, and all the
embeddings can be chosen $\Perf(Y)$-linear for their natural maps
to $Y=C/G$.
\end{thm}
\begin{proof}
The identity fixed quotient is the connected curve $Y$. Every
connected component of a nonidentity fixed quotient is a point.
By \eqref{eq:sectors}, there are exactly $e_i-1$ such points over
$y_i$, across all nonidentity classes. Match them bijectively with
the exceptional objects $M_{i,e_i-1},\ldots,M_{i,1}$, and keep the
order in \eqref{eq:ordered}. This gives the claimed decomposition.
Only a labeling of the separate point components is used; no
canonical identification of a character with a group element is
required.

For linearity, $\pi^*$ is $\Perf(Y)$-linear by the projection
formula. If $l<e_i$, then $M_{i,l}$ is annihilated by the pullback
of the maximal ideal of $y_i$, since
$\pi\inv (y_i)=e_iD_i$ scheme-theoretically. It therefore defines
a coherent equivariant kernel on
$\{y_i\}\times_Y C$. The point embedding
\[
 \Db(\pt)\longrightarrow\Db_G(C),
 \qquad V\longmapsto V\otimes_\bbC M_{i,l},
\]
is fully faithful by exceptionality. For $P\in\Perf(Y)$ the
kernel structure gives the natural derived isomorphism
\[
 (V\ox M_{i,l})\Lox \pi^*P
 \simeq (V\Lox P|_{y_i})\ox M_{i,l}.
\]
Its associativity and unit compatibilities follow from those of
derived tensor product. This proves the asserted linearity.
\end{proof}

\begin{rmk}\label{rem:linear}
For a linear action on a vector space $V$, the variety
$V^g/C_G(g)$ is irreducible, being a finite quotient of a vector
space. Thus the original and connected-component formulations
coincide in that setting. The curve counterexamples do not give
counterexamples for linear reflection-group actions.
\end{rmk}

\section{Linear conjugacy-class decompositions}
\label{sec:fusion}

Let $\pi:C\to Y=C/G$ be the quotient map. For each branch
orbit $D_i$, choose an inertia group $H_i$ and put
$e_i=|H_i|$. A functor is $\Perf(Y)$-linear if it commutes,
with coherent unit and associativity isomorphisms, with
tensoring by perfect complexes pulled back from $Y$.

\begin{prp}\label{prop:localfusion}
Let $G$ act effectively on a smooth connected quasiprojective
complex curve $C$, and let $g\ne1$. If there is a fully
faithful $\Perf(Y)$-linear exact functor
\[
 \Db(C^g/C_G(g))\longrightarrow\Db_G(C)
\]
for the natural maps to $Y=C/G$, then $|[g]\cap H_i|
 \leq\left\lfloor\frac{e_i}{2}\right\rfloor$
for every branch orbit $D_i$.
\end{prp}

\begin{proof}
Let $p$ be a point of $C^g/C_G(g)$ mapping to $y_i$, and
let $E$ be the image of $\calO_p$. For every closed
point $z\ne y_i$ of $Y$, the sheaf $\calO_z$ is
perfect and $\calO_p\Lox_{Y}\calO_z=0$.
Linearity therefore gives
$E\Lox \pi^*\calO_z=0$.
Derived Nakayama at the points over $z$ shows that $E$
has no support there. Thus $E$ is supported on $D_i$.

By \eqref{eq:sectors}, there are $|[g]\cap H_i|$ source
points over $y_i$. Their images are completely orthogonal
exceptional objects supported on $D_i$.
\lemref{lem:localbound} gives the asserted inequality.
No genus assumption is used.
\end{proof}

\prpref{prop:localfusion} proves the necessary condition in
\thmref{thm:fusioncriterion}. To prove sufficiency, we first
construct ordered semisimple blocks inside a category of type $A$.

\begin{lem}
\label{lem:typeAblocks}
Let $m_1,\ldots,m_s$ be nonnegative integers with
$\Sm_{j=1}^s m_j=n\geq1$. Suppose that
$\max_j m_j\leq\left\lceil\frac n2\right\rceil$.
Then there is a semiorthogonal decomposition
\[
 \Db(\mod\,\bbC A_n)
 =
 \left\langle\calU_1,\ldots,\calU_s\right\rangle,
 \qquad
 \calU_j\simeq\Db(\bbC^{m_j}).
\]
The order of the labelled blocks may be prescribed in
advance. Blocks with $m_j=0$ are omitted.
\end{lem}

\begin{proof}
We first colour the $n$ vertices of a path so that colour
$j$ occurs $m_j$ times and adjacent vertices have different
colours. Put $d=\lceil n/2\rceil$. List the colour blocks
in decreasing order of multiplicity and fill the positions
$1,3,\ldots,2d-1,\;2,4,\ldots,2\lfloor n/2\rfloor$.
Occurrences of a colour lying entirely in odd positions or
entirely in even positions are nonadjacent. At most one
colour crosses between the two sets of positions. If it
starts at position $2a+1$ and occupies the first $b$ even
positions, its multiplicity is $d-a+b$. A crossing colour
has multiplicity strictly less than $d$: otherwise, an
earlier colour would already occupy all $d$ odd positions.
Hence $b<a$, so its even positions lie strictly before its
odd positions and no two occurrences are adjacent.

Write $c(v)$ for the colour of vertex $v$, with colours
ordered by the prescribed block order. Orient each edge
from the endpoint of smaller colour to the endpoint of
larger colour, obtaining a quiver $Q$. Let $S_v$ be its
vertex simples in $\rep_\bbC Q$. For distinct vertices,
\[
 \Hom(S_u,S_v)=0,\qquad
 \dim_\bbC\Ext^1(S_u,S_v)
 =
 \#\{\text{arrows }u\longrightarrow v\},
\]
and all higher Ext groups vanish. Thus the simples of a
fixed colour are completely orthogonal exceptional objects.
If $b>a$, then
$\RHom(S_v,S_u)=0$, whenever $c(v)=b, c(u)=a$.
The simples generate $\Db(\rep_\bbC Q)$, so grouping
them by colour gives the required ordered blocks.

Any two orientations of a path are connected by source
and sink reflections. Their derived reflection functors
are equivalences; see \cite[Section~3.1]{Ladkani}.
Transporting the blocks along such an equivalence proves
the assertion for $\Db(\mod\,\bbC A_n)$.
\end{proof}

The preceding two results give the proof of our third main theorem.

\begin{proof}[Proof of \thmref{thm:fusioncriterion}]
The implication from (1) to (2) follows from
\prpref{prop:localfusion}.

Conversely, suppose \eqref{eq:fusioncriterion} holds.
Fix a total order $[g_1],\ldots,[g_s]$ of the nonidentity
conjugacy classes, and put
$m_{i,[g]}=|[g]\cap H_i|$.
By equation~\eqref{eq:count}, we have 
$\Sm_{[g]\ne[1]}m_{i,[g]}=e_i-1$.
By \prpref{prop:ordered},
\[
 \calB_i
 \simeq\Db(\mod\,\bbC A_{e_i-1}).
\]
Since
$\lceil(e_i-1)/2\rceil=\lfloor e_i/2\rfloor$,
by \lemref{lem:typeAblocks}, we obtain 
\[
 \calB_i
 =
 \left\langle
   \calB_{i,[g_1]},\ldots,\calB_{i,[g_s]}
 \right\rangle,
 \qquad
 \calB_{i,[g]}
 \simeq\Db(\bbC^{m_{i,[g]}}),
\]
in the chosen class order.

Different branch blocks are completely orthogonal.
Consequently, the subcategories
\[
 \calA_{[g]}
 =
 \Ds_i\calB_{i,[g]}
\]
satisfy $\calA_{[g]}
 \simeq\Db(\bbC^{n_g})
 \simeq\Db(C^g/C_G(g)),
 n_g=\Sm_i m_{i,[g]}$.
The local decompositions can now be put in the chosen class order:
orthogonality between distinct branch categories preserves both
the required vanishing and generation.
Together with \eqref{eq:ordered}, this gives
\[
 \Db_G(C)
 =
 \left\langle
   \pi^*\Db(Y),
   \calA_{[g_1]},\ldots,\calA_{[g_s]}
 \right\rangle.
\]

To construct the kernels, put
$T_i=\Ds_{l=1}^{e_i-1}M_{i,l}$.
Derived Morita equivalence for the strong exceptional
collection in \prpref{prop:ordered} shows that
every object of $\calB_i$ has a bounded representative
with terms in $\add (T_i)$.

The scheme-theoretic fibre over $y_i$ is $e_iD_i$.
Each $M_{i,l}$ with $l<e_i$ is annihilated by its ideal.
Hence every exceptional generator constructed above can
be represented by a bounded complex of coherent
$G$-equivariant sheaves on
$F_i=\{y_i\}\times_Y C$.
For a fixed class $[g]$, the bijection \eqref{eq:sectors}
identifies its $m_{i,[g]}$ source points over $y_i$.
Match them with the exceptional generators of
$\calB_{i,[g]}$. Placing the corresponding complexes
on the components $\{p\}\times_Y C$ gives a kernel on
$(C^g/C_G(g))\times_Y C$. Its Fourier--Mukai functor sends
the point simples to the chosen completely orthogonal
exceptional generators and is therefore fully faithful.

For a source point $p$ over $y_i$, denote its assigned
generator by $E_p$. The projection formula gives, for
$P\in\Perf(Y)$,
\[
 E_p\Lox \pi^*P
 \simeq
 E_p\Lox_{\bbC}P|_{y_i},
\]
where $P|_{y_i}$ denotes derived restriction.
The unit and associativity compatibilities follow from
derived tensor product. Thus the point embeddings are
$\Perf(Y)$-linear. The identity embedding $\pi^*$ is
linear by the projection formula.
\end{proof}

\begin{cor}\label{prop:nofusion}
Suppose that $|[g]\cap H_i|\leq1$ for every branch inertia
group $H_i$ and every nonidentity conjugacy class $[g]$.
Then \cnjref{conj:original} holds, with
$\Perf(Y)$-linear embeddings given by kernels on the
natural fibre products. In particular, this applies to
finite abelian groups acting effectively on smooth
connected quasiprojective complex curves.
\end{cor}

\begin{proof}
Every branch inertia order is at least two, so the
hypothesis implies \eqref{eq:fusioncriterion}.
Apply \thmref{thm:fusioncriterion}.
For an abelian group, every conjugacy class is a singleton.
\end{proof}

In this construction the exceptional generators of each branch block
are transported by a derived equivalence. The resulting blocks need
not be generated by subsets of the original collection $M_{i,l}$.

\section{Orthogonal capacity for genus-zero quotients}
\label{sec:literature}

Throughout this section $C$ is projective and $Y=C/G\simeq\bbP^1$.
The linear decomposition criterion of \thmref{thm:fusioncriterion}
already applies without a genus restriction. Here we study the different
question of the largest point category admitting an arbitrary fully
faithful embedding. Set
\[
 s=\Sm_i(e_i-1),\qquad
 \et=2-\Sm_i(1-e_i\inv ).
\]

\subsection{Bounds from the two Euler forms}
\begin{lem}\label{lem:skewbound}
Let $\calT$ have an Euler pairing on
$V=K_0(\calT)\ox\bbQ$, with $\dim V=N$.
If the alternating form $\om(x,y)=\chi(x,y)-\chi(y,x)$ has rank $2r$,
then $\oc(\calT)\leq N-r$. The positive inertia index of
$\chi+\chi^t$ is also an upper bound for $\oc(\calT)$.
\end{lem}
\begin{proof}
The classes of $m$ completely orthogonal exceptional objects have Euler
matrix $I_m$, so they are independent. Their span is totally isotropic
for $\om$. Its dimension is at most
$\dim\rad (\om)+r=N-r$, by passage to the symplectic
quotient. On the same span the symmetric form has matrix $2I_m$, proving
the second assertion.
\end{proof}

\begin{prp}\label{prop:genuszero}
The following holds:
\[
 \max\{1,b,\Sm_i\lfloor e_i/2\rfloor\}
 \leq\oc(\Db_G(C))\leq
 \begin{cases}s+1&\et>0,\\s&\et\leq0.\end{cases}
\]
If all inertia orders are two and $b\geq4$, then
$\oc(\Db_G(C))=b$.
\end{prp}
\begin{proof}
Every equivariant line bundle is exceptional, since its self-Ext is
\[
H^*(C,\calO_C)^G=H^*(\bbP^1,\calO).
\]
\lemref{lem:localbound} gives the torsion lower bound. Put
$L_i=\calO_C(D_i)$ with its natural equivariant structure.
For $i\ne j$ and $y_i=\pi(D_i)$,
\[
 (\pi_*\calO_C(D_j-D_i))^G=\calO_Y(-y_i).
\]
Indeed an invariant rational function has valuation divisible by $e_j$
at $D_j$, so allowing a pole of order one permits no pole downstairs;
the required zero at $D_i$ forces a zero at $y_i$. Consequently
$\RHom_G(L_i,L_j)=R\Gamma(\bbP^1,\calO(-1))=0$.
Thus the $b$ line bundles are completely orthogonal.

For the upper bound, let $\de=[\pi^*\calO_y]$ at a nonbranch
point. The local computations give
$\chi(F,\de)=\rank F=-\chi(\de,F)$, first for
vector bundles and then for all coherent sheaves by separating their
torsion parts. Hence $\de$ is in the radical of $\chi+\chi^t$.
At the $i$th orbit retain the character simples $S_{i,j}$ for
$0\leq j\leq e_i-2$. Together with $[\calO_C],\de$, these form
a basis of $K_0$: \prpref{prop:ordered} gives the basis using
$M_{i,l}$, and $[M_{i,l}]=\Sm_{j=0}^{l-1}[S_{i,j}]$ is a triangular
change of basis. For
$v=a[\calO_C]+d\de+\Sm_{i,j}b_{i,j}[S_{i,j}]$,
\[
 \chi(v,v)=a^2+a\Sm_i b_{i,0}
 +\Sm_i\left(\Sm_{j=0}^{e_i-2}b_{i,j}^2
              -\Sm_{j=0}^{e_i-3}b_{i,j}b_{i,j+1}\right).
\]
Here $\chi(\calO_C,S_{i,j})=\de_{j0}$ and
$\chi(S_{i,j},\calO_C)=0$ for the retained indices, by the local
resolutions in Section~\ref{sec:capacity}.
Let $A_i$ be the Cartan matrix of type $A_{e_i-1}$ and
$u_i=(1,0,\ldots,0)^t$. Since
$u_i^tA_i\inv u_i=(e_i-1)/e_i$, completing the square gives
\[
 \chi(v,v)=\frac12\Sm_i
 (b_i+aA_i\inv u_i)^tA_i(b_i+aA_i\inv u_i)+\frac\et2a^2.
\]
The positive inertia index is $s+1$ if $\et>0$, and $s$ otherwise.
Apply \lemref{lem:skewbound}. If all $e_i=2$ and $b\geq4$, then
$s=b$ and $\et=2-b/2\leq0$; the line bundles above attain the bound.
\end{proof}

In particular, a hyperelliptic curve of genus $g\geq1$, with its
hyperelliptic involution, has orthogonal capacity $2g+2$.
The upper bound applies to all exceptional objects, including vector
bundles of arbitrary rank. Notice also that, for $i\ne j$,
$\calO_C(D_i-D_j)$ is a line bundle completely orthogonal to
$\calO_C$. Thus the two-sided orthogonal to a line bundle cannot
in general be reduced to torsion sheaves.

\subsection{Tree models and three domestic types}
\begin{prp}\label{prop:treecapacity}
Let $Q$ be an orientation of a finite simple tree $T$. Write $N$ for
its number of vertices, $\nu(T)$ for its maximum matching size, and
$\al(T)$ for its maximum independent set size. Then
\[
 \oc\bigl(\Db(\mod\,\bbC Q)\bigr)
 =N-\nu(T)=\al(T).
\]
\end{prp}
\begin{proof}
In the simple-module basis, $\chi-\chi^t$ is the skew adjacency matrix
of $Q$, up to sign. Its rank is $2\nu(T)$. For completeness, choose a
leaf $v$ and its neighbour $u$. Simultaneous row and column operations
split off the invertible alternating block on $u,v$; the remaining
matrix is that of $T-\{u,v\}$. A maximum matching can be chosen to
contain $uv$: if $u$ is matched elsewhere, replace that edge by $uv$;
if $u$ is unmatched, then $v$ is also unmatched and adjoining $uv$
contradicts maximality.
Thus both rank and twice the matching size obey the same recurrence.
Isolated vertices contribute zero, proving the formula by induction
on forests. \lemref{lem:skewbound} now gives the upper bound
$N-\nu(T)$. Since a tree is bipartite, the matching--vertex-cover
theorem gives $\al(T)=N-\nu(T)$. The simple modules indexed by a
maximum independent set are exceptional and pairwise completely
orthogonal: there are no arrows between their vertices, and the path
algebra is hereditary. They attain the bound.
\end{proof}

\begin{cor}\label{cor:domesticcapacity}
For ramification types $(2,3,3)$, $(2,3,4)$ and $(2,3,5)$, respectively,
the orthogonal capacities of $\Db_G(C)$ are $4$, $5$ and $5$.
\end{cor}
\begin{proof}
The effective tame quotient stack $[C/G]$ is the weighted projective
line with the indicated stacky points and weights. Indeed it has trivial
generic stabilizer, and locally at a branch point it is
$[\Spec \bbC[[t]]/\mu_{e_i}]$ over the parameter $t^{e_i}$;
these are precisely the root-stack charts along the marked points.
The domestic derived equivalences of Geigle--Lenzing
\cite[Section~5.4.1]{GL} identify its derived category with that of a
hereditary algebra of type $\widetilde E_6$, $\widetilde E_7$ or
$\widetilde E_8$, respectively. Apply \prpref{prop:treecapacity}:
\[
\begin{array}{c|c|c|c|c}
\text{weights}&\text{tree}&N&\nu&N-\nu\\\hline
(2,3,3)&\widetilde E_6&7&3&4\\
(2,3,4)&\widetilde E_7&8&3&5\\
(2,3,5)&\widetilde E_8&9&4&5
\end{array}
\]
These trees have three arms of lengths $(2,2,2)$, $(1,3,3)$ and
$(1,2,5)$. Their independent sets attain the displayed sizes. Matchings
of sizes $3,3,4$ are immediate along the arms; in the middle case
there is no perfect matching, since deleting the centre leaves three
odd arms and the centre can be matched to only one of them.
Transporting the corresponding simple modules gives the required
exceptional objects. No classification of higher-rank bundles is needed.
\end{proof}

\section{Further questions}
\label{sec:questions}
\thmref{thm:fusioncriterion} gives a criterion for linear
conjugacy-class decompositions on every smooth curve. For arbitrary
point-category embeddings over a genus-zero quotient, the preceding
section gives exact values only for the specified ramification types.
We leave the following questions for further study.
\begin{enumerate}
\item Determine $\oc(\Db_G(C))$ for the remaining ramification types
over $\bbP^1$, in particular $(3,3,3)$, $(2,4,4)$ and $(2,3,6)$,
and for mixed wild types not covered by \prpref{prop:genuszero}.
\item When do the symmetric and alternating Euler bounds admit
simultaneous realization by completely orthogonal exceptional objects?
Numerical Euler vanishing alone does not imply Hom--Ext vanishing.
\item Describe the two-sided orthogonal to an exceptional vector bundle,
including its vector-bundle objects, and obtain geometric representatives
for maximal collections transported from hereditary models.
\end{enumerate}

\enlargethispage{2\baselineskip}

\end{document}